\documentclass[psamsfonts, reqno, 12pt]{amsart}

\usepackage{amsmath,amsfonts,amsthm,amssymb}

\newtheorem{thm}{Theorem}
\newtheorem{lem}[thm]{Lemma}

\theoremstyle{definition}

\newcommand{\ass}{\textrm{Ass}}
\newcommand{\spec}{\textrm{Spec}}
\newcommand{\kdim}{\textrm{kdim}}

\def\({\left(}
\def\){\right)}

\title[]{A directional uniformity of periodic
point distribution and mixing}
\author{Richard Miles and Thomas Ward}
\date{Draft - \today}

\address{School of Mathematics, University of East Anglia,
Norwich, NR4 7TJ, UK}
\email{r.miles@uea.ac.uk}
\address{School of Mathematics, University of East Anglia,
Norwich, NR4 7TJ, UK}
\email{t.ward@uea.ac.uk}

\subjclass[2000]{37C25, 37C40, 37C35}

\begin{document}

\begin{abstract}
For mixing~$\mathbb Z^d$-actions generated by commuting
automorphisms of a compact abelian group, we investigate the
directional uniformity of the rate of periodic point
distribution and mixing. When each of these automorphisms has
finite entropy, it is shown that directional mixing and
directional convergence of the uniform measure supported on
periodic points to Haar measure occurs at a uniform rate
independent of the direction.
\end{abstract}

\maketitle

\section{Introduction}

It is well-known that, under mild hypotheses, sufficiently
smooth functions mix at an exponential rate, and periodic point
measures become uniformly distributed on smooth functions at an
exponential rate, for dynamical systems with hyperbolic
behaviour or comparable regularity properties. For example,
Bowen~\cite[1.26]{MR0442989} shows an `exponential cluster
property', that Anosov diffeomorphisms preserving a smooth
measure mix Lipschitz functions exponentially fast, and
Lind~\cite{MR684244} shows similar properties for H{\"o}lder
functions on quasihyperbolic toral automorphisms. On compact
groups, smoothness conditions can be phrased in terms of how
well a function can be approximated by a function with finitely
supported Fourier transform (that is, by trigonometric
polynomials). Thus for a group automorphism~$\alpha:X\to X$ of
a compact metric abelian group~$X$, and an exhaustive
increasing sequence~$H_1\subset H_2\subset\cdots$ of finite
subsets of the character group~$\widehat{X}$, the rate of
mixing and the rate of uniform distribution of periodic points
amount to the existence of functions~$\phi$ and~$\psi$,
with~$\phi(k)\to\infty$ and~$\psi(k)\to\infty$ as~$k\to\infty$,
such that
\begin{enumerate}
\item $H_k\cap\widehat{\alpha}^kH_k=\{0\}$ for~$\vert n\vert>\phi(k)$
(a rate of mixing), and
\item $H_k\cap(\widehat{\alpha}^k-1)H_k=\{0\}$ for~$\vert n\vert>\psi(k)$
(a rate of equidistribution of periodic points).
\end{enumerate}
Statement~(i) gives a class of functions~$\mathcal C(X)$ with
prescribed decay of Fourier coefficients,
 a rate function~$\phi'=o(1)$, and~$C=C(f,g)$ for which
\begin{equation}\label{andwhenshegetstheresheknows}
f,g\in\mathcal C(X)\implies
\left\vert\int f(x)g(\alpha^nx){\rm d}\mu(x)
-\int f{\rm d}\mu\int g{\rm d}\mu\right\vert<C\phi'(n),
\end{equation}
where~$\mu$ denotes Haar measure on~$X$. Statement~(ii) gives a
rate function~$\psi'=o(1)$ and a constant~$C=C(f,\alpha)$ for
which
\begin{equation}\label{andshesbuyingastairwaytoheaven}
f\in\mathcal C(X)\implies\left\vert\int f{\rm d}\mu_n-\int f{\rm d}\mu\right\vert<C\psi'(n),
\end{equation}
where~$\mu_n$ denotes Haar measure on the subgroup of points
fixed by~$\alpha^n$. For a given group~$X$, the class of test
functions on which the mixing and uniform distribution may be
seen depends, via the exhaustive sequence, on the
functions~$\phi$ and~$\psi$. For explicit calculations of this
sort when~$X=\mathbb T^d$ is the torus, and~$\mathcal{C}(X)$ is
a class of H{\"o}lder functions, see
Lind~\cite[Sect.~4]{MR684244}.

Our interest here is in commuting automorphisms with finite
entropy, together defining an algebraic~$\mathbb
Z^d$-action~$\alpha$ (an \emph{entropy rank one action} in the
sense of Einsiedler and Lind~\cite{MR2031042}); this means
that~$\alpha$ is a homomorphism from~$\mathbb Z^d$ to the group
of continuous automorphisms of~$X$. We write~$\alpha^{\mathbf
n}$ for the automorphism~$\alpha(\mathbf n)$. Examples include
commuting toral automorphisms, Ledrappier's
example~\cite{MR512106}, the (invertible extension of
the)~$\times2,\times3$ system, and many others (Schmidt's
monograph~\cite{MR1345152} describes many dynamical properties
of these systems).

A non-uniform rate of mixing, or a non-uniform rate of
convergence of periodic point measures, for a~$\mathbb
Z^d$-action~$\alpha$, corresponds to the
statements~\eqref{andwhenshegetstheresheknows}
and~\eqref{andshesbuyingastairwaytoheaven} respectively for
each of the maps~$\alpha^{\mathbf n}$ with~$\mathbf n\neq0$.
The uniformity of the title amounts to asking if, having fixed
an appropriate exhaustion~$H_1\subset H_2\subset\cdots$ of the
character group~$\widehat{X}$, there is a uniform way to choose
the functions~$\phi$ and~$\psi$ witnessing a {directional
uniformity} in mixing and in the distribution of periodic
points. We show that the decay functions~$\phi'$ and~$\psi'$
can be chosen so that they depend only on the distance from the
origin in~$\mathbb Z^d$, by proving the following theorem.

\begin{thm}\label{main}
Let~$(X,\alpha)$ be a mixing entropy rank one~$\mathbb
Z^d$-action by automorphisms of a compact abelian group~$X$
satisfying the Descending Chain Condition on
closed~$\alpha$-invariant subgroups. Write~$\mu$ for Haar
measure on~$X$ and~$\mu_{\mathbf n}$ for Haar measure on the
subgroup of points fixed by the automorphism~$\alpha^{\mathbf
n}$. Then there is a class of smooth functions~$\mathcal C(X)$
strictly containing the trigonometric polynomials, and rate
functions~$\phi'=o(1)$, $\psi'=o(1)$, such that, for
any~$f,g\in\mathcal C(X)$,
\[
\left\vert\int f(x)g(\alpha^{\mathbf n}x){\rm d}\mu(x)
-\int f{\rm d}\mu\int g{\rm d}\mu\right\vert<C(f,g)\phi'(\Vert\mathbf n\Vert),
\]
and
\[
\left\vert\int f{\rm d}\mu_{\mathbf n}
-\int f{\rm d}\mu\right\vert<C(f)\psi'(\Vert\mathbf n\Vert).
\]
\end{thm}

In addition to the motivation already given, this
question arose as a result of our
paper~\cite{MR2350424}, where it is shown that for a large
class of such systems there is a uniform lower bound to the
exponential rate of growth in the number of periodic points,
and the papers~\cite{MR2308145} and~\cite{MR2279271}, in which
more subtle directionally uniform bounds and counts for
periodic points are found.
Because of the diversity of underlying
compact groups, and our main interest in uniformity, we have not
delved into the articulation between the growth in the
exhaustive sequence (measuring the smoothness of the
function class) and the permitted growth in the
control functions~$\phi$ and~$\psi$ (measuring the
rate of mixing or of equidistribution of periodic points).

Our methods combine the formalism introduced by Kitchens and
Schmidt in~\cite{MR1036904}, Diophantine arguments, and
Einsiedler and Lind's adelic Lyapunov vectors for entropy rank
one actions~\cite{MR2031042}.

\section{Algebraic $\mathbb{Z}^d$-actions}

Following Kitchens and Schmidt, we exploit the correspondence
between an algebraic~$\mathbb Z^d$-action~$\alpha$ by
automorphisms of a compact abelian metric group~$X$ and a
module over the ring of Laurent
polynomials~$R_d=\mathbb{Z}[u_1^{\pm 1},\dots,u_d^{\pm 1}]$.
This is achieved by identifying each dual
automorphism~$\widehat{\alpha}^{\mathbf{n}}$ with
multiplication by $u^{\mathbf{n}}=u_1^{n_1}\cdots u_d^{n_d}$,
then extending this in a natural way to polynomials. As a
result, attention may be restricted to a fixed~$R_d$-module~$M$
with dynamical properties of~$\alpha$ translated into algebraic
properties of~$M$. For example, the Descending Chain Condition
on closed~$\alpha$-invariant subgroups of~$X$ corresponds
to~$M$ being Noetherian. The mixing property for~$\alpha$
translates to multiplication by $u^{k\mathbf{n}}$ being
injective on $M$ for all $k\in\mathbb{N}$ and all
$\mathbf{n}\neq 0$. These two properties will be assumed
throughout. A full introduction to algebraic
$\mathbb{Z}^d$-actions and the correspondence just described is
given in Schmidt's monograph~\cite{MR1345152}.

Some especially useful algebraic machinery is available
when~$\alpha$ has entropy rank one (that is, when each element
of the action has finite topological entropy). Since~$M$ is
assumed to be Noetherian, it has a finite set of associated
prime ideals~$\ass(M)\subset\spec(R_d)$. Furthermore,
since~$\alpha$ is mixing and of entropy rank one, for
each~$\mathfrak{p}\in\ass(M)$, the module~$R_d/\mathfrak{p}$
has Krull dimension one, written~$\kdim(R_d/\mathfrak{p})=1$
(see~\cite[Prop.~6.1]{MR2031042}
and~\cite[Lem.~2.3]{MR2308145}). Therefore, the field of
fractions of~$R_d/\mathfrak{p}$ is a global field that we
denote by~$\mathbb{K}(\mathfrak{p})$.
Let~$\mathcal{P}(\mathbb{K}(\mathfrak{p}))$ denote the set of
places of~$\mathbb{K}(\mathfrak{p})$, $|\cdot|_v$ the absolute
value corresponding to the place~$v$, and set
\[
S(\mathfrak{p})=\{v\in\mathcal{P}(\mathbb{K}(\mathfrak{p}))\mid\,
|R_d/\mathfrak{p}|_v\mbox{ is an unbounded subset of }\mathbb{R}\},
\]
which is a finite set because~$R_d/\mathfrak{p}$ is finitely
generated. Following Einsiedler and Lind, associate
to~$\mathfrak{p}$ the list of \emph{Lyapunov vectors}
\[
\mathcal{L}(\mathfrak{p})=\{\boldsymbol\ell_v=(\log|\pi(u_1)|_v,\dots,\log|\pi(u_d)|_v)\mid v\in S(\mathfrak{p})\},
\]
where $\pi:R_d\rightarrow R_d/\mathfrak{p}$ denotes the usual
quotient map.

In what follows, our approach is to prove an algebraic version
of Theorem~\ref{main} for a module of the
form~$R_d/\mathfrak{p}$, and then build up to a general
Noetherian module~$M$ using standard methods
(see~\cite{MR1248915} for example).

\section{Two uniformities}

Let~$M$ be an~$R_d$-module, and suppose
that~$\(H^M_k\)_{k\geqslant 1}$ is an increasing sequence of
finite subsets of~$M$ with~$\bigcup_{k=1}^{\infty}H^M_k=M$
(that is, an exhaustive sequence). Let~$\phi_M$ and~$\psi_M$ be
functions (to be chosen later)
with~$\phi_M(k),\psi_M(k)\rightarrow\infty$
as~$k\rightarrow\infty$. We are interested in the following two
properties of~$M$.
\begin{enumerate}
\item[I:] $H^M_k\cap\mathbf u^{\mathbf{n}}H^M_k=\{0\}$ for
    all~$\mathbf{n}\in\mathbb{Z}^d$
    with~$\Vert\mathbf{n}\Vert>\phi_M(k)$ (directional uniformity of mixing),
\item[II:] $H^M_k\cap(\mathbf u^{\mathbf{n}}-1)H^M_k=\{0\}$
    for all~$\mathbf{n}\in\mathbb{Z}^d$
    with~$\Vert\mathbf{n}\Vert>\psi_M(k)$ (directional uniformity of distribution of periodic points).
\end{enumerate}

\begin{thm}\label{girl_youre_my_sunshine}
Let~$\mathfrak{p}\subset R_d$ be a prime ideal
with~$\kdim(R_d/\mathfrak{p})=1$, and suppose
that~$\theta(k)\nearrow\infty$ as~$k\rightarrow\infty$. If the
module~$M=R_d/\mathfrak{p}$ corresponds to a mixing action,
then there exists an exhaustive increasing
sequence~$\(H^M_k\)_{k\geqslant 1}$ of finite subsets of~$M$,
and a constant~$B>0$, such that Property~{\rm I} is satisfied
for
\[
\phi_M(k)=B\log \theta(k).
\]
\end{thm}

The proof of Theorem~\ref{girl_youre_my_sunshine} is
facilitated by the following result, which is adapted
from~\cite{MR2350424}.

\begin{lem}\label{dont_you_change_on_me}
If the prime ideal~$\mathfrak{p}\subset R_d$
has~$\kdim(R_d/\mathfrak{p})=1$, and the
module~$R_d/\mathfrak{p}$ corresponds to a mixing
algebraic~$\mathbb{Z}^d$-action, then the set of Lyapunov
vectors~$\mathcal{L}(\mathfrak{p})$ spans~$\mathbb{R}^d$.
\end{lem}

\begin{proof}
This can be seen using properties of the directional entropy
function~$h:\mathbb{R}^d\rightarrow \mathbb{R}$
(see~\cite{MR1355295} and~\cite{MR1869066}). The proof
of~\cite[Th.~1.1]{MR2350424} shows that~$h$ is bounded away
from zero when the action is mixing.
If~$\mathcal{L}(\mathfrak{p})$ does not span~$\mathbb{R}^d$,
then there
exists~$\mathbf{w}\in\mathsf{S}_{d-1}=\{\mathbf{z}\in\mathbb{R}^d\mid\,\Vert\mathbf{z}\Vert=1\}$
such that
\[
h(\mathbf{w})=\sum_{v\in V}\max\{\boldsymbol\ell_v\cdot\mathbf{w},0\}=0,
\]
giving an immediate contradiction.
\end{proof}

\begin{proof}[Proof of Theorem~\ref{girl_youre_my_sunshine}]
Write~$\widehat{\mathbf n}=\mathbf n/\Vert \mathbf n\Vert$ for
any non-zero integer vector~$\mathbf{n}$. We claim that there
is a constant~$C>0$ such that given any non-zero
vector~$\mathbf{n}\in\mathbb{Z}^d$, there exists~$w\in
S(\mathfrak{p})$ such
that~$|\boldsymbol\ell_w\cdot\widehat{\mathbf{n}}|> C$. If this
were not the case, then compactness of~$\mathsf{S}_{d-1}$ would
give a point~$\mathbf{z}$ in~$\mathsf{S}_{d-1}$ such
that~$\sum_{v\in
S(\mathfrak{p})}|\boldsymbol\ell_v\cdot\mathbf{z}|=0$
(since~$\mathbf{w}\mapsto\sum_{v\in
S(\mathfrak{p})}|\boldsymbol\ell_v\cdot \mathbf{w}|$ is
continuous on~$\mathsf{S}_{d-1}$), meaning
that~$\mathcal{L}(\mathfrak{p})$ would lie in the subspace
orthogonal to~$\mathbf{z}$, contradicting
Lemma~\ref{dont_you_change_on_me}.

Set
\[
H^M_k=\{a\in M\mid\theta(k)^{-1}\leqslant|a|_v\leqslant\theta(k)\mbox{ for all }v\in S(\mathfrak{p})\}\cup\{0\}.
\]
and note that~$\(H^M_k\)_{k\geqslant 1}$ is an increasing
exhaustive sequence of finite subsets of~$M$
since~$\theta(k)\nearrow\infty$ as~$k\to\infty$.

Given~$\mathbf{n}\in\mathbb{Z}^d$
with~$\Vert\mathbf{n}\Vert>\phi_M(k)$, there
exists~$v\in S(\mathfrak{p})$ such
that $|\boldsymbol\ell_v\cdot\widehat{\mathbf{n}}|>C$.
Let~$a\in H^M_k$ be non-zero.
If~$\boldsymbol\ell_v\cdot\widehat{\mathbf{n}}<0$, then
\begin{eqnarray*}
|\pi(\mathbf u^{\mathbf{n}})a|_v
& = & \exp(\Vert\mathbf{n}\Vert
\boldsymbol\ell_v\cdot\widehat{\mathbf{n}})|a|_v\\
& < & \exp\left(\frac{-C\log\theta(k)}{C}\right)\theta(k)^{1/2}\\
& = &\theta(k)^{-1/2}.
\end{eqnarray*}
On the other hand, if~$\boldsymbol\ell_v\cdot\widehat{\mathbf{n}}>0$, then
\begin{eqnarray*}
|\pi(\mathbf u^{\mathbf{n}})a|_v
& = & \exp(\Vert\mathbf{n}\Vert
\boldsymbol\ell_v\cdot\widehat{\mathbf{n}})|a|_v\\
& > & \exp\left(\frac{C\log\theta(k)}{C}\right)\theta(k)^{-1/2}\\
& = & \theta(k)^{1/2}.
\end{eqnarray*}
Hence,~$u^{\mathbf{n}}a\not\in H^M_k$, and the statement of the
theorem follows by setting~$B=\frac{1}{C}$.
\end{proof}

We now turn our attention to Property~II. For a mixing action
arising from a Noetherian module~$M$, an essential consequence
of the entropy rank one assumption is that for each~$\mathbf
n\in\mathbb Z^d$, the set of points fixed by~$\alpha^{\mathbf
n}$ is finite, and the cardinality of this set is equal
to~$\vert M/(\mathbf u^{\mathbf n}-1)M\vert$ by duality.

\begin{thm}\label{keep_it_going}
Let~$\mathfrak{p}\subset R_d$ be a prime ideal
with~$\kdim(R_d/\mathfrak{p})=1$, and suppose
that~$\theta(k)\nearrow\infty$ as~$k\rightarrow\infty$. If the
module~$M=R_d/\mathfrak{p}$ corresponds to a mixing action,
then there exists an increasing exhaustive
sequence~$H=\(H^M_k\)_{k\geqslant 1}$ of finite subsets of~$M$,
and constants~$\sigma, A, C_1, C_2>0$ such that Property~{\rm
II} is satisfied for
\[
\psi_M(k)=\max\left\{C_1, \frac{\sigma+1}{C_2}\log\left(\frac{\theta(k)}{(A^{\sigma}/2)^{1/(\sigma+1)}}\right)\right\}.
\]
\end{thm}

\begin{proof}
Just as in the proof of Theorem~\ref{girl_youre_my_sunshine},
there is a constant~$C>0$ such that given any
non-zero~$\mathbf{n}\in\mathbb{Z}^d$, there exists~$w\in
S(\mathfrak p)$ such that
\[
|\boldsymbol\ell_w\cdot\widehat{\mathbf{n}}|>C.
\]
Without loss of generality, we may always choose~$w$ such
that
\[
\boldsymbol\ell_w\cdot\widehat{\mathbf{n}}>C.
\]
For, given~$w\in S(\mathfrak p)$ such
that~$\boldsymbol\ell_w\cdot\widehat{\mathbf{n}}<-C$, we can
consider~$\prod_{v\in S(\mathfrak p)}|\pi(\mathbf
u^\mathbf{n})|_v$ as follows: Since~$\pi(\mathbf u^\mathbf{n})$
is a unit in~$M$, the product formula implies that
\[
\prod_{v\in S(\mathfrak p)\setminus\{w\}}|\pi(\mathbf u^\mathbf{n})|_v = |\pi(\mathbf u^\mathbf{n})|_w^{-1}.
\]
Hence, for some~$v\in S(\mathfrak p)\setminus\{w\}$ it follows
that
\[
|\pi(\mathbf u^\mathbf{n})|_v > |\pi(\mathbf u^\mathbf{n})|_w^{-1/\sigma},
\]
where~$\sigma=|S(\mathfrak p)|-1$. Therefore,
\[
\exp(\Vert\mathbf{n}\Vert\boldsymbol\ell_v\cdot\widehat{\mathbf{n}})>\exp(-\Vert\mathbf{n}\Vert\boldsymbol\ell_w\cdot\widehat{\mathbf{n}}/\sigma).
\]
Hence~$\boldsymbol\ell_v\cdot\widehat{\mathbf{n}}>C/\sigma$, and we simply need to replace~$C$ by~$C/\sigma$.

As before, set
\[
H^M_k=\{a\in M\mid\theta(k)^{-1}\leqslant|a|_v\leqslant\theta(k)\mbox{ for all }v\in S(\mathfrak{p})\}\cup\{0\},
\]
which again defines an increasing exhaustive sequence
since~$\theta(k)\nearrow\infty$ as~$k\to\infty$.
Fix~$\epsilon>0$ and let~$\mathbf{n}\in\mathbb{Z}^d$
satisfy~$\Vert\mathbf{n}\Vert>\psi_M(k)$, where in the
definition of~$\psi_M$,~$\sigma=|S(\mathfrak
p)|-1$,~$C_2=C-\epsilon$, and the constants~$A$ and~$C_1$ are
to be specified later. Let~$a\in H^M_k$ be non-zero and suppose
that~$(\mathbf u^{\mathbf{n}}-1)a\in H^M_k$. This implies
that~$|\pi(\mathbf u^{\mathbf{n}}-1)a|_w<\theta(k)$, so
\begin{equation}\label{make_it_all_worthwhile}
|a|_w
<\frac{\theta(k)}{|\pi(\mathbf u^{\mathbf{n}})-1|_w}
<\frac{2\theta(k)}{|\pi(\mathbf u^{\mathbf{n}})|_w}.
\end{equation}
By the product
formula~$\prod_{v\in\mathcal{P}(\mathbb{K}(\mathfrak{p}))}|a|_v=1$,
so
\[
\prod_{v\in S(\mathfrak p)\setminus\{w\}}|a|_v=|a|_w^{-1}\prod_{v\in\mathcal{P}(\mathbb{K}(\mathfrak{p}))
\setminus S(\mathfrak p)}|a|_v^{-1}\geqslant|a|_w^{-1},
\]
as~$|a|_v\leqslant 1$ for
all~$v\in\mathcal{P}(\mathbb{K}(\mathfrak{p}))\setminus
S(\mathfrak p)$. Thus at least one~$v\in S(\mathfrak
p)\setminus\{w\}$ satisfies
\begin{equation}\label{make_all_the_pain_stop}
|a|_v\geqslant|a|_w^{-1/\sigma}>\left(\frac{|\pi(\mathbf u^{\mathbf{n}})|_w}{2\theta(k)}\right)^{1/\sigma},
\end{equation}
by~\eqref{make_it_all_worthwhile}.
If~$v$ is archimedean, then by Baker's Theorem~\cite{MR0258756}, there exist constants~$A,B>0$ such that
\[
|\pi(\mathbf u^{\mathbf{n}})-1|_v\geqslant\frac{A}{\max_{1\leqslant i\leqslant d}\{n_i\}^B}.
\]
If~$v$ is non-archimedean, a similar bound holds by Yu's
Theorem~\cite{MR1055245}. In both the archimedean and
non-archimedean cases, given the ideal~$\mathfrak p$, the
constants arising can (in principle) be computed explicitly.
Combining these bounds with~\eqref{make_all_the_pain_stop}
gives
\begin{eqnarray}
\nonumber
|\pi(\mathbf u^{\mathbf{n}}-1)a|_v
=
|\pi(\mathbf u^{\mathbf{n}}-1)|_v|a|_v\negmedspace\negmedspace\negmedspace
& \geqslant &\negmedspace\negmedspace\negmedspace
\frac{A|\pi(\mathbf u^{\mathbf{n}})|_w^{1/\sigma}}{\max_{1\leqslant i\leqslant d}\{n_i\}^B(2\theta(k))^{1/\sigma}}\\
\nonumber\negmedspace\negmedspace\negmedspace
& = &\negmedspace\negmedspace\negmedspace
\frac{A\exp(\Vert\mathbf{n}\Vert\ell_w\cdot\widehat{\mathbf{n}}/\sigma)}{\max_{1\leqslant i\leqslant d}\{n_i\}^B(2\theta(k))^{1/\sigma}}\\
\nonumber\negmedspace\negmedspace\negmedspace
& \geqslant &\negmedspace\negmedspace\negmedspace
\frac{A\exp(\Vert\mathbf{n}\Vert C/\sigma)}{\max_{1\leqslant i\leqslant d}\{n_i\}^B(2\theta(k))^{1/\sigma}}\\
\negmedspace\negmedspace\negmedspace& \geqslant &\negmedspace\negmedspace\negmedspace
\frac{A}{(2\theta(k))^{1/\sigma}}\exp\(\frac{(C-\epsilon)\Vert\mathbf{n}\Vert}{\sigma}\),\label{just_like_a_river}
\end{eqnarray}
provided that~$\Vert\mathbf{n}\Vert$ is large enough
to ensure that
\[
\max_{1\leqslant i\leqslant d}\{n_i\}^B\leqslant\exp(\Vert\mathbf{n}\Vert\epsilon/\sigma).
\]
We may ensure this by a suitable choice of~$C_1=C_1(\epsilon)$ in the definition
of~$\psi_M(k)$, since~$\Vert\mathbf{n}\Vert>\psi_M(k)$.
Furthermore, since~$\Vert\mathbf{n}\Vert>\psi_M(k)$, the right-hand side
of~\eqref{just_like_a_river} is strictly greater than
\[
\frac{A}{(2\theta(k))^{1/\sigma}}
\left(\frac{\theta(k)}{(A^\sigma/2)^{1/(\sigma+1)}}\right)^{1+1/\sigma}
=
\theta(k),
\]
so~$(\mathbf u^{\mathbf{n}}-1)a\not\in H^M_k$, disagreeing with our
contrary assumption which therefore must have been false.
\end{proof}

Theorems~\ref{girl_youre_my_sunshine} and~\ref{keep_it_going}
describe (in an opaque form) uniformity in rate of mixing and
in the distribution of periodic points respectively for cyclic
systems -- those corresponding to cyclic~$R_d$-modules. As
usual, we need arguments from commutative algebra to build up
to a more general picture. The next lemma allows the two
properties to be inherited by a suitable extension of one
action by another.

\begin{lem}\label{i_love_you_just_the_way_you_are}
Let~$L,M$ be~$R_d$-modules with~$L\subset M$, and suppose
that both~$L$ and~$M/L$ are mixing.
\begin{enumerate}
\item If Property~{\rm I} holds for~$L$ and for~$M/L$, then
    there is an appropriate increasing exhaustive
    sequence~$\(H^M_k\)_{k\geqslant 1}$ and function
\[
\phi_M(k)=\max\{\phi_L(k),\phi_{M/L}(k)\},
\]
such that it also holds for~$M$.
\item If Property~{\rm II} holds for~$L$ and for~$M/L$, then there is an
appropriate increasing exhaustive sequence~$\(H^M_k\)_{k\geqslant 1}$ and function
\[
\psi_M(k)=\max\{\psi_L(k),\psi_{M/L}(k)\},
\]
such that it also holds for~$M$.
\end{enumerate}
\end{lem}

\begin{proof}
For each~$k\in\mathbb{N}$, let
\[
H_k^M=\bigcup_{x\in K\cap\pi^{-1}(H_k^{M/L})}x+H_k^L,
\]
where~$\pi:M\rightarrow M/L$ is the natural
quotient map of~$R_d$-modules, and~$K$ is a set of coset
representatives containing~$0$. By construction, each~$H_M^k$ is finite
and~$\bigcup_{k=1}^{\infty}H^M_k=M$.

\noindent(i) Suppose that Property~I is violated for
some~$\mathbf{n}$ with~$\Vert\mathbf{n}\Vert>\phi_M(k)$. Then
there exist~$w,x\in K\cap\pi^{-1}(H_k^{M/L})$ and~$g,h\in
H_k^L$ such that
\begin{equation}\label{keep_love_flowing}
\mathbf u^{\mathbf{n}}(x+h)=w+g\neq 0.
\end{equation}
In particular,~$u^{\mathbf{n}}\pi(x)=\pi(w)$,
which means that~$\pi(w)=0$ since Property~I holds for~$M/L$
by hypothesis. Therefore,~$w=0$ by our choice of~$K$, and~$\pi(x)=0$
as multiplication by~$\mathbf u^{\mathbf{n}}$ is an automorphism of~$M/L$.
It follows that~$x=0$ by our choice of~$K$ and so~\eqref{keep_love_flowing}
implies that~$\mathbf u^{\mathbf{n}}h=g$ meaning that~$g=0$,
as Property~I holds for~$L$. So,~$w+g=0$, contradicting~\eqref{keep_love_flowing}.

\noindent(ii) Suppose that Property~II is violated for
some~$\mathbf{n}$ with~$\Vert\mathbf{n}\Vert>\psi_M(k)$. Then
there exist~$x\in K$,~$w\in K\cap\pi^{-1}(H_k^{M/L})$,~$g\in
H_k^L$ and~$h\in L$ such that
\begin{equation}\label{dont_let_our_world_stop}
(\mathbf u^\mathbf{n}-1)(x+h)=w+g\neq 0.
\end{equation}
Thus~$(\mathbf u^{\mathbf{n}}-1)\pi(x)=\pi(w)$, which means
that~$\pi(w)=0$ since Property~II holds for~$M/L$. This
forces~$\pi(x)=0$ as multiplication by~$(\mathbf
u^{\mathbf{n}}-1)$ is injective on~$M/L$ (since~$M/L$
corresponds to a mixing system). Therefore~$x=0$ by our choice
of~$K$, and so~\eqref{dont_let_our_world_stop} implies
that~$(\mathbf u^\mathbf{n}-1)h=g$ meaning~$g=0$, as
Property~II holds for~$L$. So~$w+g=0$,
contradicting~\eqref{dont_let_our_world_stop}.
\end{proof}

The next lemma shows that both properties are inherited when
passing to factors of systems (factors of algebraic~$\mathbb Z^d$-actions
correspond to submodules under duality).

\begin{lem}\label{girl_youre_my_lucky_star}
Let~$L,M$ be~$R_d$-modules with~$L\subset M$ and suppose that~$M$ is mixing.
\begin{enumerate}
\item If Property~{\rm I} holds for~$M$ then there is an appropriate increasing sequence~$\(H^L_k\)_{k\geqslant 1}$ and function~$\phi_L(k)=\phi_M(k)$ such that it also holds for~$L$.
\item If Property~{\rm II} holds for~$M$ then there is an appropriate increasing exhaustive sequence~$\(H^L_k\)_{k\geqslant 1}$ and function~$\psi_L(k)=\psi_M(k)$ such that it also holds for~$L$.
\end{enumerate}
\end{lem}

\begin{proof}
For each~$k\in\mathbb{N}$
let~$H_k^L=H_k^M\cap L$. Then~$\bigcup_{k=1}^{\infty}H^L_k=L$ and each~$H_k^L$ is finite.

\noindent(i) For~$\mathbf{n}\in\mathbb{Z}^d$
with~$\Vert\mathbf{n}\Vert>\phi_L(k)$,
\begin{eqnarray*}
H_k^L \cap\mathbf u^{\mathbf{n}}H_k^L
& = & H_k^M\cap L  \cap\mathbf u^{\mathbf{n}}(H_k^M\cap L)\\
& = & H_k^M\cap L  \cap\mathbf u^{\mathbf{n}}H_k^M
\cap\mathbf u^{\mathbf{n}}L,
\end{eqnarray*}
since multiplication by~$u^{\mathbf{n}}$ is injective.
Furthermore, the right-hand side is~$\{0\}$, as Property~I holds for~$M$.

\noindent(ii) Similarly, for~$\mathbf{n}\in\mathbb{Z}^d$ with~$\Vert\mathbf{n}\Vert>\psi_L(k)$,
\begin{eqnarray*}
H_k^L \cap (\mathbf u^{\mathbf{n}}-1)H_k^L
& = & H_k^M\cap L  \cap (\mathbf u^{\mathbf{n}}-1)(H_k^M\cap L)\\
& = & H_k^M\cap L  \cap (\mathbf u^{\mathbf{n}}-1)H_k^M \cap (\mathbf u^{\mathbf{n}}-1)L,
\end{eqnarray*}
as multiplication by~$\mathbf u^{\mathbf{n}}-1$ is injective
by the mixing assumption. Furthermore, the right-hand side is~$\{0\}$,
as Property~II holds for~$M$.
\end{proof}

We are now ready to pass the two uniformity properties up
from cyclic modules to Noetherian modules.

\begin{thm}\label{i_was_living_in_darkness}
Let~$M$ be a Noetherian~$R_d$-module corresponding to a mixing
algebraic $\mathbb{Z}^d$-action of entropy rank one, and
suppose~$\theta(k)\nearrow\infty$ as~$k\rightarrow\infty$. Then
there is an increasing exhaustive
sequence~$\(H^M_k\)_{k\geqslant 1}$ of finite subsets of~$M$,
and there are constants~$B,C>0$, such that Property~{\rm I} is
satisfied for
\[
\phi_M(k)=B\log \theta(k),
\]
and Property~{\rm II} is satisfied for
\[
\psi_M(k)=C\log \theta(k).
\]
\end{thm}

\begin{proof}
Write~$\ass(M)=\{\mathfrak{p}_1,\dots,\mathfrak{p}_r\}$ and note that $\kdim(R_d/\mathfrak{p}_i)=1$ for each~$1\leqslant i\leqslant r$ by the mixing and entropy rank one assumptions.
By~\cite[Cor.~6.3]{MR1345152} or~\cite[Sect.~4]{MR1248915},~$M$ embeds in a module of the
form~$M'=\bigoplus_{i=1}^r M(i)$, where each~$R_d$-module~$M(i)$ for~$1\leqslant i\leqslant r$, has a prime filtration of the form
\begin{equation}\label{dont_change_on_me}
\{0\}=N_0^{(i)}\subset N_1^{(i)}\subset\dots\subset N_{s(i)}^{(i)}=M(i),
\end{equation}
with~$N_j^{(i)}/N_{j-1}^{(i)}\cong R_d/\mathfrak{p}_i$
for all~$1\leqslant j\leqslant s(i)$.
Each of these modules is mixing by~\cite[Th.~6.5]{MR1345152}.

We first consider Property~{\rm I}. For each module~$M(i)$,
using Theorem~\ref{girl_youre_my_sunshine},
Lemma~\ref{i_love_you_just_the_way_you_are},
and induction on the prime filtration~\eqref{dont_change_on_me}, we may find an
increasing exhaustive
sequence~$\(H^{M(i)}_k\)_{k\geqslant 1}$ of finite subsets of~$M(i)$
such that Property~I is satisfied for
\[
\phi_{M(i)}(k)=B_i\log \theta(k),
\]
where~$B_i>0$ is the constant appearing in Theorem~\ref{girl_youre_my_sunshine}
(which follows from the proof of Lemma~\ref{i_love_you_just_the_way_you_are}).
For each~$k\in\mathbb{N}$, set
\[
H_k^{M'}=\bigoplus_{i=1}^r H^{M(i)}_k
\]
and
\[
\phi_{M'}(k)=B\log\theta(k),
\]
where~$B=\max_{1\leqslant i\leqslant r}\{B_i\}$.
Therefore, Property~I holds for~$M'$
and the required result follows
by applying Lemma~\ref{girl_youre_my_lucky_star}.

Property~{\rm II} is obtained in an analogous way, noting
that~$\psi_M(k)$ can be replaced by~$\psi_M(k)=C\log\theta(k)$
for a suitably large choice of~$C$ in
Theorem~\ref{keep_it_going}.
\end{proof}

\section{Remarks}
(1) Theorem~\ref{i_was_living_in_darkness} gives
Theorem~\ref{main} simply by translation: the class of
functions~$\mathcal{C}(X)$ is defined by choosing a rate of
decay for the size of coefficients in the Fourier expansion
outside~$H_k^M$ so rapid that the sum over~$M\setminus H_k$
is~$o(1)$ in~$k$, choosing~$\theta$, and then
computing~$\phi_M'$ and~$\psi_M'$.

(2) Throughout, the function~$\theta$ can be chosen
arbitrarily. We have left it in the statements to facilitate
future more explicit estimates for specific classes of compact
groups.

(3) It seems possible that the uniformity in mixing could be
    present in higher entropy ranks. We initially attempted to
    prove this using adelic amoebas~\cite{MR2289207} in place
    of Lyapunov vectors. For a mixing action of higher entropy
    rank, an unpublished argument due to Einsiedler enables one
    to see that the adelic amoeba spans~$\mathbb{R}^d$, just as
    the set of Lyapunov vectors does for an entropy rank one
    action. However, finding a suitable exhaustive sequence in
    the dual module based on this appears to be rather
    problematic. Notably, however, one only needs to consider
    an action corresponding to a cyclic module as the method of
    passing up to Noetherian modules used here (Lemmas~\ref{i_love_you_just_the_way_you_are}
    and~\ref{girl_youre_my_lucky_star}) works for all
    entropy ranks.


\end{document}